\documentclass{amsart}
\usepackage{macros}

\title{Zariski density of crystalline points}

\author{Gebhard B\"ockle}
\address{Ruprecht-Karls-Universit\"{a}t Heidelberg}
\email{gebhard.boeckle@iwr.uni-heidelberg.de}

\author{Ashwin Iyengar}
\address{Johns Hopkins University}
\email{iyengar@jhu.edu}

\author{Vytautas Pa\v{s}k\={u}nas}
\address{Universit\"{a}t Duisburg-Essen}
\email{paskunas@uni-due.de}

\date{\today.}

\begin{document}

\maketitle

\begin{abstract} We show that crystalline points are Zariski
dense in the deformation space of a representation 
of the absolute Galois group of a $p$-adic field.  We also show that 
these points are  dense in the subspace parameterizing 
deformations with  determinant equal to a fixed crystalline character. 
Our proof is purely local and works for all $p$-adic fields and all residual Galois
representations.
\end{abstract}

\tableofcontents

\section{Introduction}

Fix a finite extension $F/\bbQ_p$ and a further large finite extension $L/\Qp$ with ring of integers $\cO \subset L$, uniformizer $\varpi$, and residue field $k$.  If $\rhobar: G_F \to \GL_d(k)$ is a continuous representation of the absolute Galois group of $F$, then let $R^\square_{\rhobar}$ denote the framed deformation ring of $\rhobar$ and write $\fX^\square_{\rhobar} = (\Spf R^\square_{\rhobar})^{\rig}$ for its rigid generic fiber.

\begin{theorem}\label{intro1}
The set $\mathcal S_{\mathrm{cr}}$ of $x \in \fX^\square_{\rhobar}$ whose associated $p$-adic Galois
 representation $\rho_x$ is crystalline with regular Hodge--Tate weights 
 is Zariski dense in $\fX^\square_{\rhobar}$. 
\end{theorem}

By \textit{Zariski dense} we mean that any rigid analytic function on $\fX^\square_{\rhobar}$ which vanishes at every point in $\cS_{\mathrm{cr}}$ is identically zero.

Let us outline the proof of the theorem.
By reducing the question to an explicit combinatorial problem we show that if $\rhobar$ is absolutely irreducible
and $\psi: G_F\rightarrow \OO^{\times}$ is a crystalline character lifting $\det \rhobar$ then there is a crystalline lift $\rho$ of $\rhobar$
with regular Hodge--Tate weights and determinant $\psi$. For general $\rhobar$ and crystalline $\psi$, we use an inductive procedure based on a result of Emerton--Gee \cite{EGStack} on extensions of crystalline representations. This procedure follows their proof that every $\rhobar$ admits a crystalline lift, but is refined to control the determinant. Our main innovation is in the irreducible case.

Then using the description of the irreducible components 
of $\fX^\square_{\rhobar}$ in our recent paper \cite{BIP}, we deduce that $\mathcal S_{\mathrm{cr}}$ 
meets every irreducible component of $\fX^\square_{\rhobar}$ non-trivially. On the other hand
we already know that the closure of $\mathcal S_{\mathrm{cr}}$ is a union of some subset of the irreducible components of $\fX^\square_{\rhobar}$. This is proved by \textit{infinite fern} arguments by Chenevier \cite{ChenAnn} if $F=\Qp$ and $\rhobar$ is absolutely irreducible, Nakamura \cite{Nak} if  $F$ is arbitrary and $\rhobar$ has scalar endomorphisms, and AI \cite{Iye} in general. The paper \cite{Iye} makes use of recent work of Breuil, Hellmann and Schraen \cite{BHS_IHES,BHS_Inv}.
Hence the closure is equal to the union of all irreducible components and we are done.

We also prove an analog of \autoref{intro1} with a 
fixed determinant. If $\psi: G_F \to \cO^\times$ is a crystalline character lifting $\det\rhobar$ then we write $R_{\rhobar}^{\square,\psi}$ for the quotient of $R_{\rhobar}^{\square}$ parameterizing framed deformations of $\rhobar$ with determinant $\psi$
 and $\fX_{\rhobar}^{\square,\psi}$ for its rigid generic fiber. 

\begin{thm}\label{intro2} The set $\mathcal S_{\mathrm{cr}}^{\psi}$ of $x \in \fX^{\square,\psi}_{\rhobar}$ whose associated $p$-adic Galois
 representation $\rho_x$ is crystalline with regular Hodge--Tate weights 
 is Zariski dense in $\fX^{\square,\psi}_{\rhobar}$. 
 \end{thm}

In \cite{BIP} we have shown that $\fX^{\square,\psi}_{\rhobar}$ is irreducible and as already explained above, we  show that 
$\mathcal S_{\mathrm{cr}}^{\psi}$ is non-empty by building on the results of Emerton--Gee \cite{EGStack}.
Unfortunately, the infinite fern arguments 
are not readily available in the literature in the
fixed determinant context. Instead, we prove, by checking that the argument of \cite{Iye} carries over,  that if we replace $\mathcal S_{\mathrm{cr}}$ by a subset  $\mathcal S_{\mathrm{cr}}^{\mathrm{cn}}$ obtained by imposing a congruence condition on the Hodge--Tate weights then 
the closure of $\mathcal S_{\mathrm{cr}}^{\mathrm{cn}}$ is still a union of irreducible components of $\fX^{\square}_{\rhobar}$. The set $\mathcal S_{\mathrm{cr}}^{\mathrm{cn}}$ is contained in the subset of $\fX^{\square}_{\rhobar}$ corresponding to 
 twists of  $\rho_x$ for $x\in  \mathcal S_{\mathrm{cr}}^{\psi}$ by crystalline characters, which are trivial modulo 
$\varpi$. This allows us to deduce that the closure of  $\mathcal S_{\mathrm{cr}}^{\psi}$ in $\fX^{\square,\psi}_{\rhobar}$ cannot have positive codimension, thus finishing  the proof of \autoref{intro2}.

As we explain in \autoref{Spec} density results
in the rigid analytic generic fibre imply 
density results in the scheme theoretic generic fibre and we obtain: 
\begin{cor}
The set $\mathcal S_{\mathrm{cr}}$ is Zariski dense in $\Spec R^\square_{\rhobar}[1/p]$ and in $\Spec R^\square_{\rhobar}$, and the same holds for $\mathcal S_{\mathrm{cr}}^{\psi}$ in $\Spec R^{\square,\psi}_{\rhobar}[1/p]$ and in $\Spec R^{\square,\psi}_{\rhobar}$.
\end{cor}

Our theorems offer a possibility of first proving 
a statement of interest for crystalline representations
lifting $\rhobar$ and then extending it to all the
representations by Zariski density. Such arguments were 
first employed by Colmez \cite{colmez} and Kisin
\cite{kisin} to show that every $2$-dimensional
irreducible $p$-adic representation of $G_{\Qp}$ 
lies in the image of the Colmez's Montreal functor. 
This  is an important result in the $p$-adic Langlands
correspondence for $\GL_2(\Qp)$ and has motivated 
further research on the  question of density of
crystalline points. 

Let us review some of the past work on this problem.
The results of \cite{colmez} and \cite{kisin} were proved for $d=2$ under 
assumptions on $\rhobar$ and $p$; the latter were subsequently removed in \cite{Bp3} when $F=\mathbb Q_3$, \cite{che_unpublished}
when $F=\mathbb Q_2$ and $\rhobar$ does not have scalar semi-simplification and \cite{CDP} when $F=\mathbb Q_2$ and $\rhobar$ is trivial. 
The argument of \cite{CDP} has been generalized by AI \cite{Iye} when $\rhobar$ is trivial and either $p>d$ or $\mu_{p^{\infty}}(F)$ is trivial. Moreover, if $d=2$, $p>2$ an $F$ is a finite extension of $\Qp$ our  
\autoref{intro1} has been proved by GB--Juschka 
in \cite{BJ2}. It has already been observed 
by Chenevier and Nakamura in their original papers that 
if one imposes certain restrictions on $\rhobar$, $F$, and $d$ ($\rhobar$ absolutely irreducible and either (1)
$\rhobar \not \cong \rhobar(1)$ or (2) $p\nmid d$ and $\zeta_p\in F$) 
then either $\fX^{\square}_{\rhobar}$ has only one irreducible component or twisting by crystalline characters
lifting the trivial representation permutes the irreducible components of $\fX^{\square}_{\rhobar}$ transitively, and their results imply that the closure 
of crystalline points contain all the irreducible components. 
Finally, we have proved \autoref{intro1} under the assumption $p\nmid 2d$ in 
\cite{BIP}. The proof given there relied on patching Galois representations associated to automorphic forms, so it was global in nature; by contrast, in this work all
methods are
local.

Our \autoref{intro1} settles the question completely for all $\rhobar$ and all $p$ and all $F$.
\autoref{intro2} also settles the 
question in the fixed determinant setting.

\subsection{Notation} Let $F$ be a fixed finite extension of $\bbQ_p$ with ring integers $\OO_F$ and residue field $k_F$. We fix a uniformiser $\varpi_F$ of $F$. We also  fix an algebraic closure $\barr{F}/F$ and let $G_F := \Gal(\barr{F}/F)$ denote the absolute Galois group of $F$ with its profinite topology. We let 
$\mut(F)$ be the group of roots of unity in $F$ and let 
$\mu := \mu_{p^\infty}(F)$ denote the group of $p$-power roots of unity in $F$.

We fix an algebraic closure $\Qpbar$ of $\Qp$ and let $L\subset \Qpbar$ be a further finite extension of $\Qp$ with ring of integers $\cO \subset L$, uniformizer $\varpi \in \cO$, and residue field $k := \cO/\varpi$. We let $\Sigma=\Sigma_F$ be the set of embeddings $\sigma: F\hookrightarrow L$, we assume that $L$ is large enough, so that $|\Sigma|=[F:\Qp]$. For a tuple  $\underline{k}=(k_{\sigma})_{\sigma\in \Sigma}\in \ZZ^{\Sigma}$ we define 
a character 
$\chi_{\underline{k}}: F^{\times}\rightarrow \OO^{\times}$ by letting $\chi_{\underline{k}}(\varpi_F)=1$, and 
$\chi_{\underline{k}}(x)=\prod_{\sigma\in \Sigma} \sigma(x)^{k_{\sigma}}$ 
for all $x\in \OO_F^{\times}$. 
The Artin map of local class field theory 
$\Art_F: F^{\times}\rightarrow G_F^{\ab}$ identifies 
$G_F^{\ab}$ with the profinite completion of $F^{\times}$, 
which allows us to consider $\chi_{\underline{k}}$ 
as a character of $G_F$. In this case $\chi_{\underline{k}}$
is crystalline with Hodge--Tate weights $\underline{k}$. Conversely, if $\chi: G_F\rightarrow \Qpbar^{\times}$ is a crystalline character with Hodge--Tate weights $\underline{k}$ then $\chi$ is 
equal to a product of $\chi_{\underline{k}}$ and an unramified character, see 
\cite[Proposition B.4]{conrad_lift}. 

In our arguments we will frequently replace 
$L$ 
by a finite extension. In particular, in \autoref{sec_two} we assume that $|\Sigma_E|=[E:\Qp]$,
where $E$ is an unramified extension of $F$ of degree $d$.

Recall that the
 Hodge--Tate weights $\underline{k}(\rho)$ of a Hodge--Tate representation $\rho: G_F\rightarrow \GL_d(\Qpbar)$ is a $d$-tuple of integers
 $\underline{k}(\rho)_{\sigma} = (k_{\sigma,1}\ge k_{\sigma,2}\ge\ldots \ge k_{\sigma, d})$ for each embedding $\sigma\in \Sigma_F$ and we say that 
 $\rho$ has \textit{regular} Hodge--Tate weights if for each $\sigma\in \Sigma_F$ 
 the inequalities are strict. 
\subsection{Acknowledgements} The authors would like to thank Gabriel 
Dospinescu and Toby Gee for their comments on an earlier version of this paper. 

GB acknowledges support by Deutsche Forschungsgemeinschaft  (DFG) through CRC-TR 326 \textit{Geometry and Arithmetic of Uniformized Structures}, project number 444845124.

\section{Crystalline lifts with fixed determinant}\label{sec_two}

In this section we construct crystalline lifts of any mod $\varpi$ representation of $G_F$. The absolutely irreducible representations are induced from characters over an unramified extension, so in this case we can just lift the characters in an explicit way using some combinatorial arguments involving Hodge--Tate weights. For non-semisimple representations, we use the Emerton--Gee stack to lift extension classes.

\begin{lemma}\label{DetInduction} Let $E/F$ be a finite Galois extension
of degree $d$. Let $\rho=\Ind_{G_E}^{G_F} \theta$, where $\theta: G_E\rightarrow A^{\times}$ is a continuous character with values in the 
unit group of some topological ring $A$. Let 
$\psi:=\det \rho$. Then 
$$\psi(\Art_F(x))=\theta(\Art_E(x)), \quad \forall x\in N^E_F E^{\times}.$$
Moreover, if $\Gal(E/F)$ is cyclic and  $g\in G_F$ maps to a generator of $\Gal(E/F)$ then 
$\psi(g)= (-1)^{d-1} \theta(g^d)$. 
\end{lemma}
\begin{proof} We have $\rho|_{G_E} \cong \oplus_{\sigma\in \Gal(E/F)} \theta^{\sigma}$. Thus  $\psi(g)= \prod_{\sigma\in \Gal(E/F)} \theta^{\sigma}(g)$, for all $g\in G_E$.
We consider $\psi$ and $\theta^{\sigma}$ as characters of $G_E^{\ab}$. Then for all $y\in E^{\times}$ we have 
\begin{equation*}
\begin{split}
\psi(\Art_E(y))&=\prod_{\sigma\in \Gal(E/F)} \theta^{\sigma}(\Art_E(y))= 
\prod_{\sigma\in \Gal(E/F)} \theta(\sigma\Art_E(y)\sigma^{-1})\\
&
=
\prod_{\sigma\in \Gal(E/F)} \theta(\Art_E(\sigma(y)))=\theta(\Art_E(N^E_F(y))).
\end{split}
\end{equation*}
A justification for the second to last equality can be found at the end of \cite[Section XI.3]{serre_local}. We have a commutative diagram 
\begin{equation}\label{diagram}
\begin{tikzcd}
    E^\times \dar[swap]{N^E_F} \rar{\Art_E} & G_E^{\ab} \dar{\iota} \\
    F^\times \rar{\Art_F} & G_F^{\ab}
\end{tikzcd}
\end{equation}
where $\iota$ is induced by the inclusion $G_E\subset G_F$, see \cite[Section 2.4]{serre_brighton}. Thus if $x= N^E_F(y)$ with
$y\in E^{\times}$ then $\psi(\Art_F(x))= \psi(\Art_E(y))= \theta(\Art_E(N^E_F(y)))= \theta(\Art_E(x))$.

Let us assume that the
image of $g\in G_F$ generates $\Gal(E/F)$. Let $f\in \Ind_{G_E}^{G_F}\theta$ be the function 
with support $G_E$ satisfying $f(1)=1$. Then 
$f, g f, \ldots, g^{d-1} f$ is a basis of $\Ind_{G_E}^{G_F}\theta$ as a $k$-vector space, and since $g^d\in G_E$ we have $g (g^{d-1} f)= \theta(g^d) f$. We  may compute the determinant of the matrix of $\rho(g)$ associated to this basis to find 
that $\det \rho(g)= (-1)^{d-1} \theta(g^d)$. 
\end{proof}
\begin{lem}\label{lem_one} Let $I$ and $J$ be finite sets. Suppose 
that we are given integers $\{a_i\}_{i\in I}$ and $\{b_j\}_{j\in J}$ such that
$\sum_{i\in I} a_i = \sum_{j\in J} b_j$. 
Then there exist integers $x_{ij}$, for $(i,j)\in I\times J$ such that 
$\sum_{i\in I} x_{ij}= b_j$ for all $j\in J$ and 
$\sum_{j\in J} x_{ij}=a_i$ for all $i\in I$. 
\end{lem} 
\begin{proof} Induction on the cardinality of the set $I$. If $|I|=1$ then letting 
$x_{ij}=b_j$ for all $j\in J$ gives the unique solution. If $|I|>1$ then we choose
$i_0\in I$ and $j_0\in J$ and define $x_{i_0 j_0}=a_{i_0}$ and $x_{i_0 j}=0$ for 
$j\in J$, $j\neq j_0$. By the induction hypothesis there exist integers $x_{ij}$ for 
$(i,j)\in (I\setminus \{i_0\})\times J$ such that $\sum_{j\in J} x_{ij} = a_i$ for all $i \in I \setminus \set{i_0}$, 
$\sum_{i\in I\setminus \{i_0\}} x_{ij}= b_j$ for all $j\in J\setminus\{j_0\}$ and 
$\sum_{i\in I\setminus \{i_0\}} x_{ij_0}= b_{j_0}-a_{i_0}.$
\end{proof}

\begin{lem}\label{regular_integers} Let $I$ and $J$ be finite sets with $|J|>1$, let $m$ be a positive integer and let $C$ be a positive integer. Suppose 
that we are given integers $\{a_i\}_{i\in I}$ and $\{b_j\}_{j\in J}$ such that
$\sum_{i\in I} a_i \equiv \sum_{j\in J} b_j \pmod{m}$. 
Then there exist pairwise distinct integers $x_{ij}$, for $(i,j)\in I\times J$, 
such that 
$\sum_{i\in I} x_{ij}\equiv  b_j\pmod{m}$ for all $j\in J$ and 
$\sum_{j\in J} x_{ij}=a_i$ for all $i\in I$ and 
$|x_{ij}|>C$ for all $(i,j)\in I\times J$. 
\end{lem}
\begin{proof} We may choose integers $b_j'$ for $j\in J$ such that $\sum_{j\in J} b_j'=
\sum_{i\in I} a_i$ and $b_j'\equiv b_j \pmod{m}$ for all $j\in J$. Let $x_{ij}$ be the 
integers given by \autoref{lem_one} for $\{a_i\}_{i\in I}$ and $\{b_j'\}_{j\in J}$. We will show that by adding suitable multiples of $m$ we can ensure that $x_{ij}$ are pairwise distinct. 
If for a fixed $i$, $x_{ij}=x_{ik}$ for $j, k\in J$, $j\neq k$ then we can replace 
$(x_{ij}, x_{ik})$ with $(x_{ij}+m N, x_{ik}-mN)$ for some $N\gg 0$. 
Repeating the process we can ensure that all $x_{ij}$ are pairwise distinct for a fixed $i$. Let $j_0\in J$ be such that $x_{ij_0}> x_{ij}$ for all $j\neq j_0$ and a fixed $i$.  If $C$ is a positive real number then by replacing 
$x_{i j_0}$ with $x_{ij_0}+ m(|J|-1)N$ and 
$x_{ij}$ with $x_{ij}- m N$ for $j\neq j_0$ and $N\gg 0$ we may ensure 
that for a fixed $i$, $x_{ij}$ are all distinct and 
satisfy $|x_{ij}|>C$. If we identify $I=\{1,\ldots, n\}$, where $n=|I|$, then using this argument we may modify the integers $x_{ij}$ so that for a fixed $i$, $x_{ij}$ are all distinct
and $|x_{i+1 j}| > \max_{j\in J} |x_{ij}|$ for $1\le i \le n-1$. Hence, all $x_{ij}$ are pairwise distinct. 
\end{proof} 

\begin{lem}\label{char} Let 
$\thetabar: k_F^{\times}\rightarrow k^{\times}$
be a character. Then there exist uniquely determined integers 
$0\le b_{\sigma}\le p-1$ indexed by the embeddings $\sigma: k_F\hookrightarrow k$
such that 
$\thetabar(x)=
\prod_{\sigma} \sigma(x)^{b_{\sigma}}$ for all 
$x\in k^{\times}_F$ 
and not all $b_{\sigma}=p-1$. 
\end{lem}
\begin{proof} Let us pick an embedding 
$\sigma_0:k_F \hookrightarrow k$. Since the character
group of $k_F^{\times}$ is cyclic of order $p^f-1$ and $\sigma_0$ is a generator there exists
a unique integer $0\le b< p^f-1$ such that
$\thetabar(x)=\sigma_0(x)^b$ for all $x\in k_F^{\times}$. We may write $b= \sum_{i=0}^{f-1} b_i p^i$ with $0\le b_i \le p-1$ and not all 
$b_i=p-1$. The digits $b_i$ are uniquely determined by $b$. For $0\le i \le f-1$ let 
$\sigma_i: k_F\hookrightarrow k$ be the embedding
$\sigma_i(x):= \sigma_0(x^{p^i})$. Then $\{\sigma_i: 0\le i \le f-1\}$ is a complete set 
of embeddings of $k_F$ into $k$ and 
$\thetabar(x)=\prod_{i=0}^{f-1} \sigma_i(x)^{b_i}.$ 
\end{proof}

\begin{lem}\label{LiftTheta} Let $E$ be an unramified extension of $F$ of 
degree $d>1$. Let $\thetabar: \OO_E^{\times}\rightarrow k^{\times}$ be a character and $\psi: \OO_F^{\times}\rightarrow \OO^{\times}$ be a character 
such that $\psi(x)= \prod_{\sigma\in \Sigma_F} \sigma(x)^{a_{\sigma}}$ for $(a_{\sigma})_{\sigma \in \Sigma_F}\in \ZZ^{\Sigma_F}$. Assume that 
$\psi(x)\equiv \thetabar(x)\pmod{\varpi}$ for all $x\in 
\OO_F^{\times}$. Then there exists a character 
$\theta: \OO_E^{\times}\rightarrow \OO^{\times}$, 
$\theta(x)= \prod_{\tau\in \Sigma_E} \tau(x)^{k_{\tau}}$, 
such that the following hold:
\begin{enumerate}
\item the integers $k_{\tau}$ are pairwise distinct;
\item $\theta$ lifts $\thetabar$;
\item $\theta(x)=\psi(x)$ for all $x\in \OO_F^{\times}$.
\end{enumerate}
\end{lem}
\begin{proof} Let $F_0$ be the maximal unramified subextension of $F$ and let $E_0$ be the unramified
extension of $F_0$ of degree $d$. Then $E$ is the compositum 
of $F$ and $E_0$ and $\Sigma_E$ is in bijection with the set 
of pairs $(\sigma, \tau_0)\in \Sigma_F \times \Sigma_{E_0}$ 
such that $\sigma|_{F_0}= \tau_0|_{F_0}$. We may further 
identify $\Sigma_{E_0}$ with the set of embeddings of $k_E$ 
into $k$. Since $\thetabar$ is trivial on $1+\pp_E$, we may 
consider it as a character of $k_E^{\times}$, and thus by \autoref{char}
there exist integers $0\le b_{\tau_0}\le p-1$ for 
$\tau_0\in \Sigma_{E_0}$ such that
$\thetabar(x)=\prod_{\tau_0\in \Sigma_{E_0}}
\tau_0(x)^{b_{\tau_0}}$ for all $x\in k_E^{\times}$.
For  $\sigma_0\in \Sigma_{F_0}$ let 
$I_{\sigma_0}=\{\sigma\in\Sigma_F: \sigma|_{F_0}=\sigma_0\}$, 
let $J_{\sigma_0}=\{\tau_0\in \Sigma_{E_0}: \tau_0|_{F_0}=\sigma_0\}$ and $\Sigma_{E, \sigma_0}:=\{\tau\in \Sigma_E: \tau|_{F_0}=\sigma_0\}$. The map $\tau\mapsto
(\tau|_F, \tau|_{E_0})$ induces a bijection between $\Sigma_{E, \sigma_0}$ and $I_{\sigma_0}\times J_{\sigma_0}$. It follows from 
\autoref{char} that there exist uniquely
determined integers $0\le c_{\sigma_0}\le p-1$
indexed by $\sigma_0\in \Sigma_{F_0}$ such that
$\thetabar(x)=\prod_{\sigma_0\in \Sigma_{F_0}} \sigma_0(x)^{c_{\sigma_0}}$ for all $x\in k_F^{\times}$ and not all $c_{\sigma_0}=p-1$. 
The relation $\psi(x)\equiv \thetabar(x) \pmod{\varpi}$ for all $x\in \OO_F^{\times}$ 
and uniqueness of $c_{\sigma_0}$ imply that 
for all $\sigma_0\in \Sigma_{F_0}$ we have
$$ \sum_{\sigma\in I_{\sigma_0}} a_{\sigma}\equiv c_{\sigma_0} \equiv \sum_{\tau_0\in J_{\sigma_0}} b_{\tau_0} \pmod{p-1}.$$
Let us label the embeddings $\Sigma_{F_0}=\{\sigma_i: 1\le i\le f\}$. Using \autoref{regular_integers} for each $i$ with $1\le i \le f$ 
we may find pairwise distinct integers $k_{\tau}$ for $\tau\in \Sigma_{E, \sigma_i}$ such that $\sum_{\tau|_{F}=\sigma} k_{\tau}=a_{\sigma}$ for all $\sigma\in I_{\sigma_i}$ and $\sum_{\tau|_{E_0}=\tau_0} k_{\tau}\equiv b_{\tau_0}\pmod{p-1}$ for all $\tau_0\in J_{\sigma_i}$. Moreover, 
$\min_{\tau\in \Sigma_{E, \sigma_{i+1}}} |k_{\tau}|> \max_{\tau\in \Sigma_{E, \sigma_{i}}} |k_{\tau}|$ for $1\le i\le f-1$. In particular, the integers $k_{\tau}$ for $\tau\in \Sigma_E$ are pairwise distinct and 
the character $\theta$ satisfies conditions (2) and (3).
\end{proof}

\begin{proposition}\label{IrrCrysLift}
If $\rhobar: G_F \rightarrow \GL_d(k)$ is absolutely irreducible and $\psi: G_F \rightarrow \cO^{\times}$ is a crystalline character lifting $\det\rhobar$, then there is a crystalline lift $\rho: G_F \rightarrow \GL_d(\cO)$ of $\rhobar$ with regular Hodge--Tate weights such that $\det \rho = \psi$.
\end{proposition}
\begin{proof}
Since $\rhobar$ is absolutely irreducible, we have $\rhobar\cong \Ind_{G_E}^{G_F} \barr{\theta}$, where $E$ is an unramified extension of $F$ of degree $d$ and $\barr{\theta}: G_E\rightarrow k^{\times}$ is a character.
Since $E/F$ is unramified $N^E_F E^{\times}= \varpi_F^{d \ZZ} \OO_F^{\times}$. It follows from \autoref{DetInduction} that 
\begin{equation}\label{eq_one}
\psi(\Art_F(x))\equiv \thetabar(\Art_E(x))\pmod{\varpi}, \quad  \forall x\in \OO_F^{\times}.
\end{equation}
Let $\varphi\in G_F$ be any element 
which maps to $\Art_F(\varpi_F)$ 
in $G_F^{\ab}$. 
Since $N^E_F(\varpi_F)=\varpi_F^d$ and 
$\varpi_F$ is a uniformiser of $E$, 
it follows from the diagram \eqref{diagram} that there is 
a uniformiser $\varpi_E$ of $E$ such that 
 the image of $\varphi^d$ in $G_E^{\ab}$ is equal to $\Art_E(\varpi_E)$. Since $E/F$ is cyclic and the image of $\varphi$ generates $\Gal(E/F)$ we have 
\begin{equation}\label{eq_two}
\psi(\Art_F(\varpi_F))\equiv 
(-1)^{d-1} \thetabar(\Art_E(\varpi_E))\pmod{\varpi}
\end{equation}
by the last part of \autoref{DetInduction}.

We first lift $\barr{\theta}$ to a crystalline character $\theta$, then induce. Since $\Art_E$ induces 
an isomorphism between $G_E^{\ab}$ and the 
profinite completion of $E^{\times}$, it is enough to define $\theta(\Art_E(x))$ for $x\in E^{\times}$. 
We let
\begin{equation}\label{eq_three}
\theta(\Art_E(\varpi_E)) := (-1)^{d-1}\psi(\Art_F(\varpi_F)). 
\end{equation}
Using \eqref{eq_one} and \autoref{LiftTheta} we can find a set of pairwise distinct integers $k_\tau$ for $\tau \in \Sigma_E$ and define
    \[ \theta(\Art_E(x)) = \prod_{\tau \in \Sigma_E} \tau(x)^{k_\tau}, \quad \forall x \in \cO_E^\times \]
so that $\theta(\Art_E(x))\equiv \thetabar(\Art_E(x))\pmod{\varpi}$ for all $x\in \OO_E^{\times}$  and $\theta(\Art_E(x)) = \psi(\Art_F(x))$ for all $x \in \cO_F^\times$. It follows from 
\eqref{eq_one}, \eqref{eq_two} and \eqref{eq_three} that $\theta$ lifts $\thetabar$. By \cite[Proposition B.4]{conrad_lift} $\theta$ is crystalline with  Hodge--Tate weights $(k_{\tau})_{\tau\in \Sigma_E}$. Since $E/F$ is unramified
$\rho := \Ind_{G_E}^{G_F} \theta$ is a crystalline representation lifting $\rhobar$ and  since the integers $k_{\tau}$ are pairwise distinct  $\rho$ has regular Hodge--Tate weights, see  \cite[Corollary 7.1.2]{GHS}.
By \autoref{DetInduction}
    \begin{equation}\label{four}
    (\det\rho)(\Art_F(x)) = \theta(\Art_E(x)) = \psi(\Art_F(x)), \quad \forall  x \in \cO_F^\times 
    \end{equation}
and
    \begin{equation}\label{five}
    \begin{split}
    (\det\rho)(\Art_F(\varpi_F)) = \det \rho(\varphi)&=(-1)^{d-1}\theta(\varphi^d)\\&= (-1)^{d-1}\theta(\Art_E(\varpi_E)) = \psi(\Art_F(\varpi_F)).
    \end{split}
    \end{equation}
It follows from \eqref{four} and \eqref{five} that $\det \rho=\psi$.
\end{proof}

Now we use the Emerton--Gee stack defined in \cite{EGStack} to extend \autoref{IrrCrysLift} to all $\rhobar$.

\begin{proposition}\label{CrysLift}
For any continuous  representation $\rhobar: G_F \rightarrow \GL_d(k)$ and a crystalline character  $\psi: G_F \rightarrow \cO^{\times}$ lifting $\det\rhobar$ there is a crystalline lift $\rho: G_F \rightarrow
\GL_d(\cO)$ of $\rhobar$ (after possibly enlarging $L$) with regular Hodge--Tate weights and $\det\rho = \psi$.
\end{proposition}
\begin{proof} We argue by induction on the number of irreducible subquotients of 
$\rhobar$ by repeatedly applying \cite[Theorem 6.3.2]{EGStack} in the induction step. 
If $\rhobar$ is absolutely irreducible  then we are done by \autoref{IrrCrysLift}. For the inductive step we may write
    \[ 0 \to \rhobar_1 \to \rhobar \to \rhobar_2 \to 0 \]
with $\rhobar_2$ absolutely irreducible. By \cite[Lemma 6.5]{BIP} there is a
crystalline character $\psi_1$ lifting $\det \rhobar_1$. Then $\psi_2:= \psi \psi_1^{-1}$
is a crystalline character lifting $\det \rhobar_2$. 
By the induction  hypothesis we may assume that $\rhobar_1, \rhobar_2$ have crystalline lifts $\rho_1, \rho_2$ with regular Hodge--Tate weights satisfying $\det \rho_1= \psi_1$ and $\det \rho_2=\psi_2$. 

Let $d_1$ and $d_2$ be the dimensions of $\rhobar_1$ and $\rhobar_2$ respectively. We may choose $N\gg 0$ such that the Hodge--Tate weight of $\rho'_1:= \rho_1\otimes \chi_{\cyc}^{d_2(p-1)N}$ are all positive and the Hodge--Tate weights of
$\rho_2':=\rho_2\otimes\chi_{\cyc}^{-d_1(p-1)N}$ are all negative, where $\chi_{\cyc}$ is the $p$-adic cyclotomic character. Since $\chi_{\cyc}^{p-1}$ is a crystalline character, which is trivial modulo $\varpi$, $\rho_1'$ and $\rho_2'$ are crystalline lifts of $\rhobar_1$ 
and $\rhobar_2$ respectively. Moreover, the Hodge--Tate weights of $\rho_1'$ are \textit{slightly less} than the Hodge--Tate weights of $\rho_2'$, 
in the terminology of \cite[Section 6.3]{EGStack}. (We note that $\chi_{\cyc}$ has Hodge--Tate $1$ under our conventions, 
and $-1$ under the convention of \cite{EGStack}. This leads to the reversal of inequalities.) Further, 
$$\psi_1':= \det \rho_1'= \psi_1\chi_{\cyc}^{d_1 d_2 (p-1)N}, \quad \psi_2':=\det \rho_2'= \psi_2\chi_{\cyc}^{-d_1 d_2 (p-1)N}.$$ Hence, 
$\psi_1' \psi_2'=\psi$.
Then we are in the situation of \cite[Theorem 6.3.2]{EGStack}, which yields a lift (after possibly enlarging $L$) of $0 \to \rhobar_1 \to \rhobar \to \rhobar_2 \to 0$ given by
    \[ 0 \to \rho_1'' \to \rho \to \rho_2' \to 0 \]
    such that $\rho$ and $\rho_1''$ are crystalline and
    $\rho''_1$ has the same Hodge--Tate weights as $\rho'_1$. 
    This implies that $\psi_1'':=\det \rho_1''$ and $\psi_1'$ 
    are both crystalline with the same Hodge--Tate weights, and 
    hence $\eta:=\psi_1'(\psi_1'')^{-1}$ is an unramified character 
    by \cite[Proposition B.4]{conrad_lift}. Since both $\psi_1'$
    and $\psi_1''$ lift $\det \rhobar_1$, $\eta$ is 
    is trivial modulo $\varpi$. After 
    enlarging $L$ we may find an unramified character $\chi: G_F \rightarrow 1+\pp_L$, such that $\chi^d=\eta$.  
    Then $\rho\otimes \chi$ is a crystalline lift of $\rhobar$ with regular Hodge--Tate weights and determinant $\psi$. 
\end{proof}

\begin{cor}\label{CrysLiftMu} Let $\chi: \mu_{p^{\infty}}(F)\rightarrow \OO^{\times}$ 
be a character. Then (after possibly enlarging $L$) there exists a crystalline lift $\rho: G_F\rightarrow \GL_d(\OO)$ of $\rhobar$ 
with regular Hodge--Tate weights such that $(\det \rho)(\Art_F(x))=\chi(x)$ for all $x\in \mu_{p^{\infty}}(F)$.
\end{cor} 
\begin{proof} By \cite[Lemma 6.5]{BIP} there exists a crystalline 
character $\psi: G_F \rightarrow \OO^{\times}$ lifting $\det \rhobar$, 
such that $\psi(\Art_F(x))=\chi(x)$ for all $x\in \mu_{p^{\infty}}(F)$.
The assertion follows from \autoref{CrysLift}.
\end{proof}

\section{Zariski density}

Now fix a continuous representation $\rhobar: G_F \to \GL_d(k)$ and assume that its irreducible subquotients are absolutely irreducible; this can always be achieved by replacing $L$ with an unramified extension since the image of $\rhobar$ is finite. Let $R_{\rhobar}^\sq$ denote the universal framed deformation ring of $\rhobar$ with universal representation $\rho^{\univ}: G_F \to \GL_d(R_{\rhobar}^\sq)$.
We will first describe the irreducible components of the
rigid analytic space
$\mathfrak X^{\sq}_{\rhobar}
:=(\Spf R^{\sq}_{\rhobar})^{\rig}$ using the results of \cite{BIP}.

By local class field theory there is a group homomorphism 
	\[ \mu \hra F^\times \xra{\Art_F} G_F^{\ab} \xra{\det \rho^{\univ}} (R_{\rhobar}^{\sq})^\times \]
which induces a homomorphism of $\OO$-algebras $\cO[\mu] \to R_{\rhobar}^{\sq}$, where $\mu:=\mu_{p^{\infty}}(F)$.  Since $\cO[\mu][1/p] \cong \prod_{\chi: \mu \to \cO^\times} L$ we get a decomposition
	\[ R_{\rhobar}^\sq[1/p] \cong \prod_{\chi: \mu \to \cO^\times} R_{\rhobar}^{\sq,\chi}[1/p] \]
where $R_{\rhobar}^{\sq,\chi} = R_{\rhobar}^\sq \otimes_{\cO[\mu],\chi} \cO$ for each $\chi: \mu \to \cO^\times$. The rings 
$R_{\rhobar}^{\sq,\chi}$ are $\OO$-flat, normal, integral domains
and $R^{\sq}_{\rhobar}$ is $\OO$-flat by \cite[Theorem 1.2]{BIP}. It follows from 
\cite[Theorem 2.3.1]{conrad_irr} that the irreducible components of $\mathfrak X^{\sq}_{\rhobar}$ 
are given by $\mathfrak X^{\sq,\chi}_{\rhobar}:=(\Spf R^{\square, \chi}_{\rhobar})^{\rig}$ for 
$\chi: \mu\rightarrow \OO^{\times}$.

\begin{remark}\label{ChiComponent}
If $x \in \fX_{\rhobar}^\sq$ is a closed point then its residue field $\kappa(x)$ is a finite extension of $L$. The reduction map $R_{\rhobar}^\sq \thra \kappa(x)$ gives rise to a representation
	\[ \rho_x: G_F \xra{\rho^{\univ}} \GL_d(R_{\rhobar}^\sq) \to \GL_d(\kappa(x)). \]
If $\chi_x$ denotes the map $\mu \to G_F^{\ab} \xra{\det\rho_x} \kappa(x)^\times$ then $\chi_x$ actually lands in $\cO^\times$ and we have $x \in \fX_{\rhobar}^{\sq,\chi_x}$.
\end{remark}

\begin{definition}
A point $x \in \fX^\square_{\rhobar}$ is \textit{ regular crystalline} if $\rho_x$ is crystalline with regular Hodge--Tate weights.
\end{definition}

\begin{theorem}\label{DensityTheorem}
The subset $\mathcal S_{\mathrm{cr}}$
of regular crystalline points in
$\mathfrak{X}^\sq_{\rhobar}$ is
Zariski dense.
\end{theorem}
\begin{proof} It follows from 
 \autoref{CrysLiftMu} and 
 the decription of the irreducible components of $\mathfrak{X}^\sq_{\rhobar}$
 given above that every irreducible component contains 
 a regular crystalline point. 
 By \cite[Proposition 5.10]{Iye} the Zariski closure of the regular crystalline points in $\mathfrak{X}^\sq_{\rhobar}$ is a union of irreducible components. Thus $\mathcal S_{\mathrm{cr}}$ is Zariski dense in 
 $\mathfrak{X}^\sq_{\rhobar}$.
\end{proof}

\section{Deformation rings with fixed determinant}

Let $\psi: G_F\rightarrow \OO^{\times}$ be a crystalline character lifting 
$\det \rhobar$. It follows from \autoref{CrysLift} that (possibly after 
replacing $L$ by a finite extension) there is a crystalline lift 
$\rho: G_F\rightarrow \GL_d(\OO)$ of $\rhobar$ with regular Hodge--Tate weights $\underline{k}(\rho)$ such that $\det \rho=\psi$.

Let $\mathcal S_{\mathrm{cr}}^{\mathrm{cn}}$ be the subset of $\mathfrak X^{\square}_{\rhobar}$ consisting of $x$ such that 
$\rho_x$ is crystalline with regular Hodge--Tate weights 
$\underline k(\rho_x)$ and $$k_{\sigma, i}(\rho_x) \equiv k_{\sigma, i}(\rho)\pmod{dt}$$
for all $\sigma\in \Sigma$ and $1\le i\le d$, where $t$ is the order of 
$\mut(F)$. 

\begin{lem}\label{twist_shout} If $x\in \mathcal S_{\mathrm{cr}}^{\mathrm{cn}}$  then there is a crystalline character 
$\theta: G_{F}\rightarrow \Qpbar^{\times}$ lifting the trivial character, 
such that $\theta(\Art_F(x))=1$ for all $x\in \mut(F)$ and 
$\det (\rho_x\otimes \theta)=\psi$.
\end{lem} 
\begin{proof} Since both $\rho$ and $\rho_x$ are crystalline, the character $\eta:= \psi(\det \rho_x)^{-1}$ is crystalline with Hodge--Tate weights 
$k_{\sigma}(\eta)= \sum_{i=1}^d (k_{\sigma, i}(\rho)- k_{\sigma,i}(\rho_x))$
for $\sigma \in \Sigma$. It follows from the definition of $\mathcal S_{\mathrm{cr}}^{\mathrm{cn}} $ that 
$dt$ divides $k_{\sigma}(\eta)$. Let $\underline{k}\in \ZZ^{\Sigma}$ be the 
tuple $(k_{\sigma}(\eta)/d)_{\sigma\in \Sigma}$ then the character 
$\chi_{\underline{k}}$ is crystalline and $\eta \chi_{\underline{k}}^{-d}$ is unramified. Since $\chi_{\underline{k}}$ is a $t$-th power of a character
of $G_F$, it is trivial on the torsion subgroup of $G_F^{\ab}$ and hence 
$\chi_{\underline{k}}(g)\equiv 1 \pmod{\varpi}$ for all $g\in I_F$. Since
$\chi_{\underline{k}}(\Art_F(\varpi_F))=1$ by definition, we conclude that 
$\chi_{\underline{k}}$ lifts the trivial character. Since both $\rho_x$
and $\rho$ lift $\rhobar$, $\eta$ lifts the trivial character. In particular,
$\eta(\Art_F(\varpi_F))\in 1+\mm_{\Zpbar}$. We may choose 
$y\in 1+\mm_{\Zpbar}$ such that $y^d= \eta(\Art_F(\varpi_F))$, and 
let $\alpha: G_F\rightarrow 1+\mm_{\Zpbar}$ be an unramified character such 
that $\alpha(\Art_F(\varpi_F))=y$. Then $\theta:=\alpha \chi_{\underline{k}}$
satisfies all the conditions.
\end{proof} 

\begin{thm}\label{refined_dense} The Zariski closure of $\mathcal S_{\mathrm{cr}}^{\mathrm{cn}} $ in $\mathfrak X^{\square}_{\rhobar}$ is a union of irreducible components of $\mathfrak X^{\square}_{\rhobar}$.
\end{thm} 

Before proving the theorem let us note the following consequence. 
\begin{cor}\label{main_det} The subset $\mathcal S^{\psi}$ of $\mathfrak X^{\square, \psi}_{\rhobar}$ consisting of $x$ such that 
$\rho_x$ is crystalline with regular Hodge--Tate weights is Zariski dense
in $\mathfrak X^{\square, \psi}_{\rhobar}$.
\end{cor}
\begin{proof} Let $\mathcal X: \mathfrak A_{\OO} \rightarrow \mathrm{Sets}$ be the functor, such that $\mathcal X(A)$ is the set of characters
$\theta_A: G_F \rightarrow 1+\mm_A$, which are trivial on $\Art_F(\mut(F))$. 
This functor is pro-representable by $\OO(\mathcal X)\cong \OO\br{x_0,\ldots,  x_n}$, with $n=[F:\Qp]$. The map $$D^{\square, \psi}_{\rhobar}(A)\times \mathcal X(A)\rightarrow D^{\square}_{\rhobar}(A), \quad (\rho_A, \theta_A)\mapsto \rho_A\otimes_A \theta_A$$ induces a homomorphism of local $\OO$-algebras
$R^{\square}_{\rhobar}\rightarrow R^{\square, \psi}_{\rhobar}\wtimes_{\OO} \OO(\mathcal X)$. It follows from Corollary 5.2 and Lemma 5.3 in \cite{BIP}
that this map is finite and becomes \'etale after inverting $p$. Thus,
the rigid analytic space 
$\mathcal X^{\rig}:=(\Spf \OO(\mathcal X))^{\rig}$ is an 
open polydisc and the induced map on rigid analytic spaces
$$\varphi: \mathfrak X^{\square, \psi}_{\rhobar} \times_{\Sp L} \mathcal X^{\rig}\rightarrow \mathfrak X^{\square}_{\rhobar}$$
is finite and \'etale. 
Let $\mathcal S'$ be the subset of 
$\mathfrak X^{\square, \psi}_{\rhobar} \times_{\Sp L} \mathcal X^{\rig}$
 corresponding to pairs $(\rho', \theta)$ such that $\rho'\in \mathcal S^{\psi}$ and $\theta$ is crystalline. 
 We note that $\mathcal S'$ is non-empty, since 
it contains a point corresponding to the pair $(\rho, \Eins)$.  Let $\mathfrak V$ be 
the Zariski closure of 
$\mathcal S'$ inside 
$\mathfrak X^{\square, \psi}_{\rhobar} \times_{\Sp L} \mathcal X^{\rig}$.
We have shown in \cite[Theorem 5.6]{BIP} 
that $R^{\square, \psi}_{\rhobar}$ is a normal integral domain. Thus if 
$\mathcal S^{\psi}$ is not Zariski dense in  
$\mathfrak X^{\square, \psi}_{\rhobar}$ 
then $\mathfrak V$ has positive codimension in 
$\mathfrak X^{\square, \psi}_{\rhobar} \times_{\Sp L} \mathcal X^{\rig}$.
Since $\varphi$ is finite $\varphi(\mathfrak V)$ is a closed subset of 
$\mathfrak X^{\square}_{\rhobar}$ of positive codimension. \autoref{twist_shout} implies that $\varphi(\mathfrak V)$ contains the 
set $\mathcal S_{\mathrm{cr}}^{\mathrm{cn}} $. \autoref{refined_dense} implies that 
$\varphi(\mathfrak V)$ contains an irreducible component of $\mathfrak X^{\square}_{\rhobar}$. We have shown in \cite[Theorem 1.1]{BIP} that 
$R^{\square}_{\rhobar}$ is complete intersection, and hence equidimensional. 
Thus $\dim \varphi(\mathfrak V)= 
\dim \mathfrak X^{\square}_{\rhobar}$ yielding a contradiction. 
\end{proof}

We will now prove \autoref{refined_dense}. We follow the proof of 
\cite[Proposition 5.10]{Iye}, where the analogous statement is shown 
for the set of all regular crystalline points, i.e. without imposing the 
congruence condition on the Hodge--Tate weights. 

We recall following \cite[Section 3.3.1]{bel_che} that a subset $\mathcal S$ of a rigid analytic space $X$ \textit{accumulates} at 
$x\in X$ if there is a basis of open neighbourhoods $U$ of $x$ such that $\mathcal S \cap U$ is Zariski dense in $U$. 
Let $\mathcal W$ be a rigid analytic space over $L$ parameterizing 
the continuous characters of 
$\mathcal O_F^{\times}$.
\begin{lem}\label{cong_acc} Let $x\in \mathcal W^d$ correspond to a $d$-tuple of characters 
$(\chi_{\underline{k}_1}, \ldots, \chi_{\underline{k}_d})$ for some $\underline{k}_i\in \ZZ^{\Sigma}$ for $1\le i\le d$. 
Let $C$ be a real positive number and let $\mathcal S'$ be 
the subset of $\mathcal W^d$ corresponding to $d$-tuples of characters $(\chi_{\underline{k}'_1}, \ldots, \chi_{\underline{k}'_d})$ with $\underline{k}'_i \in \ZZ^{\Sigma}$
for $1\le i \le d$ such that $k_{\sigma,i}'-k_{\sigma, i+1}'> C$ and $k'_{\sigma,i}\equiv 
k_{\sigma,i} \pmod{dt}$ for all $\sigma\in \Sigma$ and $i\ge 1$.
Then $\mathcal S'$ accumulates at $x$.  
\end{lem} 
\begin{proof} This follows from the argument of \cite[Lemme 2.7]{Che_Hecke}.
\end{proof}

We will now prove the analog
of \cite[Proposition 5.9]{Iye}. Let
$X^{\square}_{\tri}(\rhobar)\subset \mathfrak X^{\square}_{\rhobar}\times \mathcal T^d_L$ be the trianguline variety, see \cite[Section 2.1]{BHS_Inv}. Let $\mathcal S''$ be the subset of 
$X^{\square}_{\tri}(\rhobar)$ such that $x\in \mathcal S''$
correspond to pairs $(\rho_x, \delta_x)$ with $\rho_x\in \mathcal S_{\mathrm{cr}}^{\mathrm{cn}} $, defined at the beginning of the section, and 
$\rho_x$ is $\varphi^f$-generic and noncritical, see \cite[Definition 5.5]{Iye} for these notions, as well 
for the term \textit{benign}.

\begin{prop}\label{sub} Suppose $x=(\rho_x, \delta_x)\in X^{\square}_\tri(\rhobar)$ is benign. Then $\mathcal S''$ accumulates at $x$.
\end{prop}
\begin{proof} The argument of 
\cite[Proposition 4.1.4]{BHS_IHES} goes through, the only change 
is that in the last sentence of the proof we use \autoref{cong_acc}. 
\end{proof}

\begin{proof}[Proof of \autoref{refined_dense}] Let $Z$ be an irreducible component of the closure of $\mathcal S_{\mathrm{cr}}^{\mathrm{cn}} $, and let $Z^{\sm}$ be the smooth locus inside $Z$. Since $Z^{\sm}$ is open 
and $\mathcal S_{\mathrm{cr}}^{\mathrm{cn}} $ is dense in $Z$, there is $x\in Z^{\sm}\cap \mathcal S_{\mathrm{cr}}^{\mathrm{cn}} $. 
Let $\rho_x$ be the corresponding crystalline Galois representation 
and let $\underline{k}$ be the Hodge--Tate weights of $\rho_x$. Let 
$\mathfrak X^{\square,\underline{k}}_{\rhobar, \mathrm{cr}}$ be 
the rigid analytic space associated to the crystalline deformation
ring of $\rhobar$ with fixed Hodge--Tate weights $\underline k$.
Let $Z^{\underline k}_{\mathrm{cr}}$ be an irreducible component 
of 
$\mathfrak X^{\square,\underline{k}}_{\rhobar, \mathrm{cr}}$ containing $x$.  Then $\mathfrak X^{\square,\underline{k}}_{\rhobar, \mathrm{cr}}$ is contained in $\mathcal S_{\mathrm{cr}}^{\mathrm{cn}} $ and hence 
$Z^{\underline k}_{\mathrm{cr}}$ is contained in $Z$. It follows 
from \cite[Theorem 3.3.8]{kisin_pst} that irreducible components 
of $\mathfrak X^{\square,\underline{k}}_{\rhobar, \mathrm{cr}}$
do not intersect. Thus $Z^{\underline k}_{\mathrm{cr}}$ is open 
in $\mathfrak X^{\square,\underline{k}}_{\rhobar, \mathrm{cr}}$.
By applying \cite[Lemma 4.2]{Nak} with $U=Z^{\sm}\cap Z^{\underline k}_{\mathrm{cr}}$, we deduce that there is $y\in U$, such that $\rho_y$ is benign. We note that since the Hodge--Tate weights of 
$\rho_y$ are equal to $\underline{k}$, they satisfy the congruence 
condition imposed in the definition of $\mathcal S_{\mathrm{cr}}^{\mathrm{cn}} $, and hence 
$y\in \mathcal S_{\mathrm{cr}}^{\mathrm{cn}} $. The rest of the proof of 
\cite[Proposition 5.10]{Iye} carries over verbatim by using \autoref{sub} instead of \cite[Proposition 5.9]{Iye}.
\end{proof}
\section{Density in scheme-theoretic generic fibre}\label{Spec}
Let $R$ be a complete local Noetherian $\OO$-algebra with residue field $k$. Let $\fX=(\Spf R)^{\rig}$ be the rigid analytic space 
associated to the formal scheme $\Spf R$, \cite[Section 7]{deJong}. The canonical map 
\begin{equation}\label{map}
R[1/p] \rightarrow \Gamma(\fX, \OO_{\fX})
\end{equation}
induces a bijection between the set of 
maximal ideals of $R[1/p]$ and the set of points of 
$\fX$. Moreover, if 
 $x\in \fX$ corresponds to a maximal ideal $\mm\subset R[1/p]$ then there is a natural map of local rings 
 \begin{equation}\label{map2}
 R[1/p]_{\mm}\rightarrow \OO_{\fX, x},
 \end{equation}
 compatible with \eqref{map}, which induces an isomorphism on completions, see \cite[Lemma 7.1.9]{deJong}. 
 
 \begin{lem}\label{dejong} If $\mathcal S$ is a Zariski dense subset 
 of $\fX$ then $\mathcal S$ is also dense in $\Spec R[1/p]$. Moreover, if $R$ is $\OO$-torsion free then $\mathcal S$ is dense in $\Spec R$.
 \end{lem}
 \begin{proof} The map \eqref{map2} is injective, since it induces an isomorphism on completions.
 This implies that the map \eqref{map} is 
 injective. If $f\in R[1/p]$ vanishes at all $x\in \mathcal S$ then by considering $f$ as a function on $\fX$ via \eqref{map} we deduce that $f$ is zero. 
 Hence, $\mathcal S$ is Zariski dense in $\Spec R[1/p]$.
 
 If $R$ is $\OO$-torsion free then the generic points of $\Spec R$ and
 $\Spec R[1/p]$ coincide and hence $\mathcal S$ is dense in $\Spec R$.
 \end{proof}

By \cite[Theorem 1.1 and Theorem 1.4]{BIP} $R_{\rhobar}^\square$ and $R_{\rhobar}^{\square,\psi}$ are $\cO$-torsion free, so as a direct consequence of \autoref{DensityTheorem}, \autoref{main_det}, \ and 
\autoref{dejong} we obtain: 

\begin{cor}
The set $\mathcal S_{\mathrm{cr}}$ is Zariski dense in $\Spec R^\square_{\rhobar}[1/p]$ and in $\Spec R^\square_{\rhobar}$, and the same holds for $\mathcal S_{\mathrm{cr}}^{\psi}$ in $\Spec R^{\square,\psi}_{\rhobar}[1/p]$ and in $\Spec R^{\square,\psi}_{\rhobar}$.
\end{cor}

\printbibliography

\end{document}